\documentclass[12pt]{amsart}
\usepackage{amscd,amsmath,amssymb,enumerate,amsfonts,mathrsfs,txfonts}
\usepackage{comment}
\usepackage{hyperref}
\usepackage{xcolor}
\usepackage{tikz,pgfplots,pgf}
\usepackage[all]{xy}

\tikzset{node distance=3cm, auto}
\usetikzlibrary{calc} 
\usetikzlibrary{trees} 

\usepackage{vmargin}
\setpapersize{A4}
\setmargins{2.5cm}      
{1.5cm}                        
{16.5cm}                      
{23.42cm}                   
{10pt}                          
{1cm}                           
{0pt}                          
{2cm}
\newtheorem{theorem}{Theorem}[section]

\newtheorem{proposition}[theorem]{Proposition}
\newtheorem{corollary}[theorem]{Corollary}
\theoremstyle{definition}
\newtheorem{definition}[theorem]{Definition}

\numberwithin{equation}{section}

\def\N{\mathbb{N}}
\def\R{\mathbb{R}}
\def\C{\mathbb{C}}

\def\F{\mathcal{F}}
\def\G{\mathcal{G}}
\def\H{\mathcal{H}}
\def\I{\mathcal{I}}
\def\J{\mathcal{J}}
\def\K{\mathcal{K}}
\def\L{\mathcal{L}}

\def\Nu{\mathcal{N}}

\def\Hv{\H_{v}}
\def\HvKp{\H_{v\K_p}}

\def\HvWp{\H_{v\K_{wp}}}

\def\HvKup{\H_{v\K_{\text{u}p}}}

\def\QN{\mathcal{QN}}

\def\Gv{\G_{v}}

\def\At{\mathrm{At}}
\def\conv{\mathrm{conv}}

\def\lin{\mathrm{lin}}
\def\abco{\mathrm{abco}}
\def\rank{\mathrm{rank}}

\begin{document}

\title[Weighted holomorphic mappings associated with $p$-compact type sets]{Weighted holomorphic mappings associated \\ with $p$-compact type sets}

\author[M. G. Cabrera-Padilla]{M. G. Cabrera-Padilla}
\address[M. G. Cabrera-Padilla]{Departamento de Matem\'aticas, Universidad de Almer\'ia, Ctra. de Sacramento s/n, 04120 La Ca\~{n}ada de San Urbano, Almer\'ia, Spain.}
\email{m\_gador@hotmail.com}

\author[A. Jim{\'e}nez-Vargas]{A. Jim{\'e}nez-Vargas}
\address[A. Jim{\'e}nez-Vargas]{Departamento de Matem\'aticas, Universidad de Almer\'ia, Ctra. de Sacramento s/n, 04120 La Ca\~{n}ada de San Urbano, Almer\'ia, Spain.}
\email{ajimenez@ual.es}

\author[A. Keten \c Copur]{A. Keten \c Copur}
\address[A. Keten \c Copur]{Department of Mathematics and Computer Science, Faculty of Sciences, Necmettin Erbakan University, Konya, 42090, T\"{u}rkiye.}
\email{aketen@erbakan.edu.tr}

\thanks{Research partially supported by Junta de Andaluc\'ia (grant FQM194) and by Ministerio de Ciencia e Innovaci\'{o}n (grant PID2021-122126NB-C31 funded by MICIU/AEI/10.13039/501100011033 and by ERDF/EU) for the first two authors.}

\date{\today}



\begin{abstract}
Given an open subset $U$ of a complex Banach space $E$, a weight $v$ on $U$, and a complex Banach space $F$, let $\H^\infty_v(U,F)$ denote the Banach space of all weighted holomorphic mappings $f\colon U\to F$, under the weighted supremum norm $\left\|f\right\|_v:=\sup\left\{v(x)\left\|f(x)\right\|\colon x\in U\right\}$. In this paper, we introduce and study the classes of weighted holomorphic mappings $\H^\infty_{v\K_{p}}(U,F)$ (resp., $\H^\infty_{v\K_{wp}}(U,F)$ and $\H^\infty_{v\K_{up}}(U,F)$) for which the set $(vf)(U)$ is relatively $p$-compact (resp., relatively weakly $p$-compact and relatively unconditionally $p$-compact). We prove that these mapping classes are characterized by $p$-compact (resp., weakly $p$-compact and unconditionally $p$-compact) linear operators defined on a Banach predual space of $\H^\infty_v(U)$ by linearization. 
We show that $\H^\infty_{v\K_{p}}$ (resp., $\H^\infty_{v\K_{wp}}$ and $\H^\infty_{v\K_{up}}$) is a Banach ideal of weighted holomorphic mappings which is generated by composition with the ideal of $p$-compact (resp., weakly $p$-compact and unconditionally $p$-compact) linear operators and contains the Banach ideal of all right $p$-nuclear weighted holomorphic mappings. We also prove that these weighted holomorphic mappings can be factorized through a quotient space of $l_{p^*}$, and $f\in\H^\infty_{v\K_{p}}(U,F)$ (resp., $f\in \H^\infty_{v\K_{up}}(U,F))$ if and only if its transposition $f^t$ is quasi $p$-nuclear (resp., quasi unconditionally $p$-nuclear).
\end{abstract}

\maketitle


\section*{Introduction}\label{section 0}

Grothendieck \cite{g} in 1955 characterized a relatively compact set in Banach spaces by means of sequences converging to zero, showing that compactness is also a geometric property. Inspired by Grothendieck's characterization, Sinha and Karn \cite{SinKar-02} defined the concept of a relatively (weakly) $p$-compact set, and, in parallel, a (weakly) $p$-compact linear operator. Kim \cite{Kim-14} introduced the notion of a relatively unconditionally $p$-compact set as a weaker version of a relatively $p$-compact set and a stronger version of a relatively compact set, and, in parallel, the concept of an unconditionally $p$-compact linear operator. The works \cite{SinKar-02, Kim-14} have inspired many papers exploring various aspects of linear operators associated with these sets (see \cite{DelPinSer-10, GalLasTur-12, Kim-17, Kim-19, Kim-20, LasTur-12, SinKar-08}), as well as Lipschitz operators (see \cite{AchDahTur-19, Ket-Inal}).

When $E$ and $F$ are Banach spaces and $U$ is an open subset on $E$,  linearization results assist to relate a class of holomorphic mappings defined on $U$ and valued in $F$ to the space of bounded linear operators from a certain Banach space $G(U)$ into $F$. This relation effectively connects a holomorphic mapping with a corresponding linear operator via linearization, thus allowing many properties of the linear case to be treated in the holomorphic case.

The study of weighted Banach spaces of holomorphic mappings and operators between them is a topic with a long tradition that continues to attract the attention of researchers today. The interested reader can find complete information on this important function space in the survey \cite{Bon-22} by Bonet and the references therein.

The weighted spaces of holomorphic mappings defined on an open subset of the finite dimensional space $\mathbb{C}^N$ ($N\in {\mathbb{N}}$) have been investigated in \cite{BieSum-93, BieSum-93-2, BonDomLin-01, Rub-Shie-70} while the infinite dimensional case has been addressed by \cite{Gar-Mae-Rue-00, GupBaw-16, Jorda-13}.

For this latter case, the research on holomorphic mappings with a relatively compact range was started in \cite{Muj-91} and continued in \cite{JimRuiSep-23}. A study on holomorphic mappings with a relatively $p$-compact range was conducted in \cite{Jim-23}. The weighted holomorphic mappings with a relatively compact $v$-range were briefly discussed in \cite{GupBaw-16}.

In this study, we consider classes of weighted holomorphic mappings associated to various types of $p$-compact sets, inspired by the studies \cite{GalLasTur-12, Kim-17, Kim-20, SinKar-02} and \cite{Jim-23, JimRuiSep-23} on different versions of bounded linear operators and holomorphic mappings, respectively.

This work consists of two parts divided in some subsections and is organized as follows. In Section \ref{section 1}, we provide the necessary notation and preliminary information for the study. In Section \ref{section 2}, we introduce and study classes of weighted holomorphic mappings with relatively (resp., relatively weakly, relatively unconditionally) $p$-compact $v$-range. 

To study the Banach ideal structure of these classes, we release and examine the ideal structure of weighted holomorphic mappings. We obtain important characterizations between the introduced classes of mappings and their linearizations. Using these characterizations, we derive factorization results for such classes and explore various characterizations and results connecting these mapping classes with their transpositions. 

We also deal with finite-dimensional $v$-rank and approximable weighted holomorphic mappings and show that a $p$-approximable (resp., $wp$-approximable, $up$-approximable) weighted holomorphic mapping has a relatively (resp., relatively weakly, relatively unconditionally) $p$-compact $v$-range. Finally, we present the class of right $p$-nuclear weighted holomorphic mappings and discuss the relationship between this class and its linearization. We also show that this class belongs to the class of weighted holomorphic maps with a relatively $p$-compact $v$-range.


\section{Notation and preliminaries}\label{section 1}

We will fix some notation and gather some necessary facts and concepts on both theory of weighted holomorphic mappings and theory of $p$-compact operators that will be used without any express mention along the paper.

We refer to Bierstedt and Summers \cite{BieSum-93}, Bonet, Domanski and Lindstr\"om \cite{BonDomLin-01,BonDomLin-99} and with Taskinen \cite{BonDomLinTas-98}, Gupta and Baweja \cite{GupBaw-16} and Mujica \cite{Muj-91} for the former theory, and to Delgado, Pi\~neiro and Serrano \cite{DelPinSer-10}, Galicer, Lassalle and Turco \cite{GalLasTur-12, LasTur-12}, Kim \cite{Kim-14,Kim-17,Kim-19,Kim-20} and Sinha and Karn \cite{SinKar-02,SinKar-08} for the latter theory. For the theory of operator ideals, we recommend the book \cite{Pie-80} by Pietsch.

\subsection{Notation}

As usual, $\mathbb{K}$ denotes the field $\mathbb{R}$ or $\mathbb{C}$. Let $E$ and $F$ be complex Banach spaces and let $U$ be an open subset of $E$.  Let $\H(U,F)$ be the space of all holomorphic mappings from $U$ into $F$. A \textit{weight} $v$ on $U$ is a (strictly) positive continuous function.

The \textit{space of weighted holomorphic mappings} $\H_v^\infty(U,F)$ is the Banach space of all mappings $f\in\H(U,F)$ such that   
$$
\left\|f\right\|_v:=\sup\left\{v(x)\left\|f(x)\right\|\colon x\in U\right\}<\infty ,
$$ 
endowed with the \textit{weighted supremum norm} $\left\|\cdot\right\|_v$. It is usual to write $\H^\infty_v(U)$ instead of $\H^\infty_v(U,\C)$. Moreover, $\H^\infty_{v\K}(U,F)$ and $\H^\infty_{v\K_{w}}(U,F)$ stand for the spaces of mappings $f\in\H(U,F)$ such that $vf$ has relatively compact range and relatively weakly compact range in $F$, respectively.

We denote by $\L(E,F)$ the Banach space of all bounded linear operators from $E$ into $F$, equipped with the operator canonical norm, and by $\F(E,F)$, $\K(E,F)$ and $\K_{w}(E,F)$ the subspaces of $\L(E,F)$ consisting of all finite-rank operators, compact operators and weakly compact operators from $E$ into $F$, respectively. As usual, $E^{*}$ and $B_E$ represent the topological dual space and the closed unit ball of $E$, respectively. Given a set $A\subseteq E$, $\overline{\lin}(A)$ and $\overline{\abco}(A)$ stand for the norm closed linear hull and the norm closed absolutely convex hull of $A$ in $E$. Given two normed spaces $(X,\left\|\cdot\right\|_X)$ and $(Y,\left\|\cdot\right\|_Y)$, the inequality $(X,\left\|\cdot\right\|_X)\leq (Y,\left\|\cdot\right\|_Y)$ will indicate that $X\subseteq Y$ and $\left\|x\right\|_Y\leq\left\|x\right\|_X$ for all $x\in X$.  

\subsection{Preliminaries on weighted holomorphic mappings}

Let $U$ be an open set of a complex Banach space $E$ and let $v$ be a weight on $U$. The key tool in our study is the space $\G_v^\infty(U)$, a Banach predual space of the space $\H_v^\infty(U)$. 

Following \cite{BieSum-93, GupBaw-16}, $\G_v^\infty(U)$ is the space of all linear functionals on $\H_v^\infty(U)$ whose restriction to $B_{\H_v^\infty(U)}$ is continuous for the compact-open topology. In fact, $\G_v^\infty(U)$ is a closed subspace of $\H_v^\infty(U)^*$, and the evaluation mapping $J_v\colon\H_v^\infty(U)\to\G_v^\infty(U)^*$, given by $J_v(f)(\phi)=\phi(f)$ for $\phi\in\G_v^\infty(U)$ and $f\in\H_v^\infty(U)$, is an isometric isomorphism.

For each $x\in U$, the \emph{evaluation functional} $\delta_x\colon\H^\infty_v(U)\to\mathbb{C}$, defined by $\delta_x(f)=f(x)$ for $f\in\H^\infty_v(U)$, is in $\G^\infty_v(U)$, and the map $\Delta_v\colon U\to\G_v^\infty(U)$ given by $\Delta_v(x)=\delta_x$ is in $\H_v^\infty(U,\G_v^\infty(U))$ with $\left\|\Delta_v\right\|_v\leq 1$.

By \emph{an atom of $\G^\infty_v(U)$} we mean an element of $\G^\infty_v(U)$ of the form $v(x)\delta_x$ for $x\in U$. The set of all atoms in $\G^\infty_v(U)$ will be denoted here by $\At_{\G^\infty_v(U)}$. It is known that $B_{\G^\infty_v(U)}$ coincides with $\overline{\abco}(\At_{\G^\infty_v(U)})\subseteq\H^\infty_v(U)^*$ and, consequently, $\G_v^\infty(U)$ with $\overline{\lin}(\At_{\G^\infty_v(U)})\subseteq\H_v^\infty(U)^*$. 

The space $\G_v^\infty(U)$ enjoys \emph{an universal property}: for every complex Banach space $F$ and every mapping $f\in\H_v^\infty(U,F)$, there exists a unique operator $T_f\in\L(\G_v^\infty(U),F)$ such that $T_f\circ\Delta_v=f$. Furthermore, $\left\|T_f\right\|=\left\|f\right\|_v$. As a consequence, the correspondence $f\mapsto T_f$ is an isometric isomorphism from $\H_v^\infty(U,F)$ onto $\L(\G_v^\infty(U),F)$ (resp., from $\H^\infty_{v\K}(U,F)$ onto $\K(\G_v^\infty(U),F)$, from $\H^\infty_{v\K_{w}}(U,F)$ onto $\K_w(\G_v^\infty(U),F))$. 

For each $f\in\H^\infty_v(U,F)$, the mapping $f^t\colon F^*\to\H^\infty_v(U)$, defined by $f^t(y^*)=y^*\circ f$ for all $y^*\in F^*$, is in $\L(F^*,\H^\infty_v(U))$ with $||f^t||=\left\|f\right\|_v$ and $f^t=J_v^{-1}\circ(T_f)^*$, where $(T_f)^*\colon F^*\to\G^\infty_v(U)^*$ is the adjoint operator of $T_f$.

\subsection{Preliminaries on $p$-compact operators}

Given a Banach space $E$ and $p\in [1,\infty)$, consider the Banach spaces of \textit{$p$-summable sequences, 
weakly $p$-summable sequences and unconditionally $p$-summable sequences} in $E$ defined, respectively, by
\begin{align*}
\ell_p(E)&=\left\{(x_n)\in E^\N\colon \left\|(x_n)\right\|_p:=\left(\sum_{n=1}^\infty\left\|x_n\right\|^p\right)^{\frac{1}{p}}<\infty\right\},\\
\ell^w_p(E)&=\left\{(x_n)\in E^\N\colon \left\|(x_n)\right\|^w_p:=\sup_{x^*\in B_{E^*}}\left\|(x^*(x_n))\right\|_p<\infty\right\},\\
\ell^u_p(E)&=\left\{(x_n)\in \ell^w_p(E)\colon \lim_{m\to\infty}\left\|(x_n)_{n\geq m}\right\|^w_p=0\right\}.
\end{align*}
In the case $E=\mathbb{K}$, we will write $\ell_p$, 
$\ell^w_p$ and $\ell^u_p$. The following relations are known:
$$
\left(\ell_p(E),\left\|\cdot\right\|_p\right)\leq\left(\ell^u_p(E),\left\|\cdot\right\|^w_p\right)\leq\left(\ell^w_p(E),\left\|\cdot\right\|^w_p\right).
$$
We also deal with the Banach spaces of \textit{norm null sequences} and \textit{weakly null sequences} in $E$,
\begin{align*}
c_0(E)&=\left\{(x_n)\in E^\N\colon \lim_{n\to\infty} \left\|x_n\right\|=0\right\},\\
c^w_0(E)&=\left\{(x_n)\in E^\N\colon \lim_{n\to\infty}\left|x^*(x_n)\right|=0,\; \forall x^*\in E^*\right\},
\end{align*}
equipped with the \textit{supremum norm}
$$
\left\|(x_n)\right\|_\infty=\sup_{n\in\N}\left\|x_n\right\|.
$$
For $p\in (1,\infty)$, the \textit{$p$-convex hull} of a sequence $(x_n)\in\ell_p(E)$ is defined by
$$
p\text{-}\conv(x_n)=\left\{\sum_{n=1}^\infty a_nx_n\colon (a_n)\in B_{\ell_{p^*}}\right\},
$$
where $p^*$ denotes the \textit{conjugate index of $p$}. Similarly, we set 
\begin{align*}
1\text{-}\conv(x_n)&=\left\{\sum_{n=1}^\infty a_nx_n\colon (a_n)\in B_{c_0}\right\},\ \ (x_n)\in \ell_1(E),\\
\infty\text{-}\conv(x_n)&=\left\{\sum_{n=1}^\infty a_nx_n\colon (a_n)\in B_{\ell_1}\right\},\ \ (x_n)\in c_0(E).
\end{align*}

Given $p\in [1,\infty]$, a set $K\subseteq E$ is said to be \textit{relatively $p$-compact (resp., relatively weakly $p$-compact, relatively unconditionally $p$-compact)} if there exists a sequence $(x_n)\in\ell_p(E)$ (resp., $(x_n)\in\ell^w_p(E)$, $(x_n)\in\ell^u_p(E)$) for $p\in [1,\infty)$, or $(x_n)\in c_0(E)$ (resp., $(x_n)\in c_0^w(E)$, $(x_n)\in c_0(E)$) for $p=\infty$, such that $K\subseteq p\text{-}\conv(x_n)$. 

Given Banach spaces $E,F$ and $p\in[1,\infty]$, an operator $T\in\L(E,F)$ is said to be \textit{$p$-compact (resp., weakly $p$-compact, unconditionally $p$-compact)} if $T(B_E)$ is a relatively $p$-compact (resp., relatively weakly $p$-compact, relatively unconditionally $p$-compact) subset of $F$. The space of all $p$-compact (resp., weakly $p$-compact, unconditionally $p$-compact) operators from $E$ into $F$ is denoted by $\K_p(E,F)$ (resp., $\K_{wp}(E,F)$, $\K_{up}(E,F)$). 

The sequence $(x_n)$ in the definition of a relatively $p$-compact set $K\subseteq E$ is not unique but, following to Lassalle and Turco \cite{LasTur-12} (see also \cite{DelPinSer-10}), a \textit{measure of the size of the $p$-compactness of $K$} comes given by 
$$
m_{\K_p}(K)=\left\{
\begin{array}{ll}
\inf\left\{\left\|(x_n)\right\|_p\colon (x_n)\in\ell_p(E), K\subseteq p\text{-}\conv(x_n)\right\} & \text{if}\; 1\leq p<\infty, \\
\inf\left\{\left\|(x_n)\right\|_p\colon (x_n)\in c_0(E), K\subseteq p\text{-}\conv(x_n)\right\} & \text{if}\; p=\infty .
\end{array}\right.
$$
Similarly, we can define 
$$
m_{\K_{wp}}(K)=\left\{
\begin{array}{ll}
\inf\left\{\left\|(x_n)\right\|^w_p\colon (x_n)\in\ell^w_p(E), K\subseteq p\text{-}\conv(x_n)\right\} & \text{if}\; 1\leq p<\infty, \\
\inf\left\{\left\|(x_n)\right\|_p\colon (x_n)\in c^w_0(E), K\subseteq p\text{-}\conv(x_n)\right\} & \text{if}\; p=\infty ,
\end{array}\right.
$$
and 
$$
m_{\K_{up}}(K)=\left\{
\begin{array}{ll}
\inf\left\{\left\|(x_n)\right\|^w_p\colon (x_n)\in\ell^u_p(E), K\subseteq p\text{-}\conv(x_n)\right\} & \text{if}\; 1\leq p<\infty, \\
\inf\left\{\left\|(x_n)\right\|_p\colon (x_n)\in c_0(E), K\subseteq p\text{-}\conv(x_n)\right\} & \text{if}\; p=\infty ,
\end{array}\right.
$$
whenever $K\subseteq E$ is relatively weakly $p$-compact or relatively unconditionally $p$-compact, respectively.

Considering for $i=p,wp,up$ and $p\in [1,\infty]$, the norm
$$
\left\|T\right\|_{\K_i}=m_{\K_i}(T(B_E))\qquad (T\in\K_i(E,F)),
$$
it is known that $[\K_i,\left\|\cdot\right\|_{\K_i}]$ is a Banach operator ideal (see \cite[Theorems 4.1 and 4.2]{SinKar-02}, \cite[Proposition 3.15]{DelPinSer-10}, \cite[Theorem 2.1 and Remark 2.3]{Kim-20} and \cite[Theorem 2.1]{Kim-14}), and the following relations hold whenever $1<p\leq\infty$:
\begin{align*}
[\K_p,\left\|\cdot\right\|_{\K_p}]&\leq[\K_{up},\left\|\cdot\right\|_{\K_{up}}]\leq [\K_{u\infty},\left\|\cdot\right\|_{\K_{u\infty}}]=[\K,\left\|\cdot\right\|]=[\K_\infty,\left\|\cdot\right\|_{\K_\infty}],\\
[\K_{up},\left\|\cdot\right\|_{\K_{up}}]&\leq[\K_{wp},\left\|\cdot\right\|_{\K_{wp}}]\leq [\K_w,\left\|\cdot\right\|].
\end{align*}

With these concepts, we are ready to present the classes of weighted holomorphic mappings we will be dealing with. 


\section{Results}\label{section 2}

From now on, unless otherwise stated, $E$ and $F$ will denote complex Banach spaces, $U$ will be an open subset of $E$, $v$ will be a weight on $U$ and $p\in\left[1,\infty\right]$. 

We will study different classes of weighted holomorphic mappings $f\colon U\to F$ whose $v$-range, that is, the set
$$
(vf)(U)=\left\{v(x)f(x)\colon x\in U\right\},
$$
satisfies some type of $p$-compactness. To be more precise, we will focus on the following spaces.

\begin{definition}
We define
\begin{align*}
\H_{v\K_p}^\infty(U,F)&=\left\{f\in\H_{v}^{\infty}(U,F)\colon (vf)(U) \text{ is a relatively $p$-compact subset of }F\right\},\\
\H_{v\K_{wp}}^\infty(U,F)&=\left\{f\in\H_{v}^{\infty}(U,F)\colon (vf)(U) \text{ is a relatively weakly $p$-compact subset of }F\right\},\\
\H_{v\K_{up}}^\infty(U,F)&=\left\{f\in \H_{v}^{\infty}(U,F)\colon (vf)(U) \text{ is a relatively unconditionally $p$-compact subset of }F\right\}.
\end{align*}
For $\I=\K_p,\K_{wp},\K_{up}$ and $f\in\H_{v\I}^\infty(U,F)$, we set 
$$
\left\|f\right\|_{\H_{v\I}^\infty}=m_\I((vf)(U)).
$$
We will say that $(vf)(U)$ \emph{has the $\I$-property} if $f\in\H_{v\I}^\infty(U,F)$ for $\I=\K_p,\K_{wp},\K_{up}$.
\end{definition}


Our approach in this paper covers some of these spaces because in view of \cite[p. 20]{SinKar-02} and \cite[p. 135]{Kim-14}, both relatively $\infty$-compact sets and relatively unconditionally $\infty$-compact sets in a Banach space are precisely the relatively compact sets, and thus we have the following result.

\begin{proposition}\label{new}
$(\H_{v\K_\infty}^\infty(U,F),\left\|\cdot\right\|_{\H_{v\K_\infty}^\infty})=(\H_{v\K_{u\infty}}^\infty(U,F),\left\|\cdot\right\|_{\H_{v\K_{u\infty}}^\infty})=(\H_{v\K}^\infty(U,F),\|\cdot\|_v)$. $\hfill\Box$
\end{proposition}

Considering the relations of inclusion between the three types of relatively $p$-compact sets considered, we easily get the following. 

\begin{proposition}
$[\HvKp^\infty,\left\|\cdot\right\|_{\H_{v\K_p}^\infty}]\leq [\HvKup^\infty,\left\|\cdot\right\|_{\H_{v\K_{up}}^\infty}]\leq [\HvWp^\infty,\left\|\cdot\right\|_{\H_{v\K_{wp}}^\infty}]$ for any $p\in [1,\infty]$; and $[\H_{v\K_p}^\infty,\left\|\cdot\right\|_{\H_{v\K_p}^\infty}]\leq [\H_{v\K_q}^\infty,\left\|\cdot\right\|_{\H_{v\K_q}^\infty}]$ and $
[\H_{v\K_{wp}}^\infty,\left\|\cdot\right\|_{\H_{v\K_{wp}}^\infty}]\leq [\H_{v\K_{wq}}^\infty,\left\|\cdot\right\|_{\H_{v\K_{wq}}^\infty}]$ if $1\leq p\leq q\leq\infty$. 
$\hfill\Box$
\end{proposition}


\subsection{Banach ideal property}

Our next goal is to study the Banach ideal structure of $[\H^\infty_{v\I},\left\|\cdot\right\|_{\H_{v\I}^\infty}]$ for $\I=\K_p,\K_{wp},\K_{up}$. Influenced by the concept of Banach operator ideal (see \cite{Pie-80}), the following type of ideals was considered in \cite[Definition 2.3]{CabJimRui-23} for $v=1_U$ (the function constantly $1$ on $U$). 

\begin{definition}\label{def-weighted holomorphic ideal} 
A \textit{normed (Banach) ideal of weighted holomorphic mappings} (or simply, a \textit{normed (Banach) weighted holomorphic ideal}) is an assignment $[\I^{\Hv^\infty},\left\|\cdot\right\|_{\I^{\Hv^\infty}}]$ which associates with every pair $(U,F)$, where $U$ is an open subset of a complex Banach space $E$ and $F$ is a complex Banach space, a subset $\I^{\Hv^\infty}(U,F)$ of $\Hv^\infty(U,F)$ and a function $\|\cdot\|_{\I^{\Hv^\infty}}\colon\I^{\Hv^\infty}(U,F)\to\R$ such that the following conditions are satisfied
\begin{enumerate}
\item[(P1)] $\left(\I^{\Hv^\infty}(U,F),\|\cdot\|_{\I^{\Hv^\infty}}\right)$ is a normed (Banach) space with $\|f\|_v\leq \|f\|_{\I^{\Hv^\infty}}$ if $f\in\I^{\Hv^\infty}(U,F)$, 
\item[(P2)] For $h\in\Hv^\infty(U)$ and $y\in F$, the map $h\cdot y\colon x\mapsto h(x)y$ from $U$ to $F$ is in $\I^{\Hv^\infty}(U,F)$ with $\|h\cdot y\|_{\I^{\Hv^\infty}}=\|h\|_v||y||$,
\item[(P3)] \textit{The ideal property}: if $V$ is an open subset of $E$ such that $V\subseteq U$, $h\in\H(V,U)$ with $c_v(h):=\sup_{x\in V}(v(x)/v(h(x)))<\infty$, $f\in\I^{\Hv^\infty}(U,F)$ and $S\in\L(F,G_0)$ where $G_0$ is a complex Banach space, then $S\circ f\circ h\in\I^{\Hv^\infty}(V,G_0)$ with $\|S\circ f\circ h\|_{\I^{\Hv^\infty}}\leq\left\|S\right\|\|f\|_{\I^{\Hv^\infty}}c_v(h)$.

\end{enumerate}

We must point out that the condition imposed on $h$ in (P3) appears usually in the study of composition operators between spaces of weighted holomorphic mappings (see, for example, \cite[Theorem 4]{BonDomLin-99} and \cite[Proposition 2.1]{BonDomLinTas-98}).

Given two normed weighted holomorphic ideals $[\I^{\Hv^\infty},\left\|\cdot\right\|_{\I^{\Hv^\infty}}]$ and $[\J^{\Hv^\infty},\left\|\cdot\right\|_{\J^{\Hv^\infty}}]$, we will write 
$$
[\I^{\Hv^\infty},\left\|\cdot\right\|_{\I^{\Hv^\infty}}]\leq [\J^{\Hv^\infty},\left\|\cdot\right\|_{\J^{\Hv^\infty}}]
$$
to indicate that for any complex Banach space $E$, any open subset $U$ of $E$ and any complex Banach space $F$, one has $\I^{\Hv^\infty}(U,F)\subseteq\J^{\Hv^\infty}(U,F)$ and $\left\|f\right\|_{\J^{\Hv^\infty}}\leq \left\|f\right\|_{\I^{\Hv^\infty}}$ for all $f\in\I^{\Hv^\infty}(U,F)$. 
\end{definition}

To construct new ideals of weighted holomorphic mappings, we introduce in this setting the composition method with an operator ideal. This has been a standard method to yield new ideals of function spaces in different contexts as, for example, spaces of Lipschitz functions \cite{AchDahTur-19}, spaces of polynomials and holomorphic functions \cite{AroBotPelRue-10}, spaces of polynomial and multilinear operators \cite{BotPelRue-07}, and spaces of bounded holomorphic functions \cite{CabJimRui-23}, among others. 

\begin{definition}\label{composition ideal}
Given a normed operator ideal $[\I,\|\cdot\|_\I]$, a mapping $f\in\Hv^\infty(U,F)$ \textit{belongs to the composition ideal $\I\circ\Hv^\infty$}, and we write $f\in\I\circ\Hv^\infty(U,F)$, if there are a complex Banach space $G$, an operator $T\in\I(G,F)$ and a mapping $g\in\Hv^\infty(U,G)$ such that $f=T\circ g$, that is, the following diagram commutes
$$
\begin{tikzpicture}
  \node (U) {$U$};
  \node (F) [right of=U] {$F$};
  \node (G) [below of=U] {$G$};
  \draw[->] (U) to node {$f$} (F);
  \draw[->] (U) to node [swap] {$g$} (G);
  \draw[->] (G) to node [swap] {$T$} (F);
\end{tikzpicture}
$$
If $f\in\I\circ\Hv^\infty(U,F)$, we denote $\|f\|_{\I\circ\Hv^\infty}=\inf\left\{\|T\|_\I\|g\|_v\right\}$, where the infimum is extended over all factorizations of $f$ as above.
\end{definition}

It is easy to show that $\I\circ\Hv^\infty$ is in fact a weighted holomorphic ideal.

\begin{proposition}\label{ideal comp}
Let $[\I,\|\cdot\|_\I]$ be a normed operator ideal. Then $[\I\circ\Hv^\infty,\|\cdot\|_{\I\circ\Hv^\infty}]$ is a normed weighted holomorphic ideal.
\end{proposition}

\begin{proof}
Let $f\in\I\circ\Hv^\infty(U,F)$ and assume that $f=T\circ g$ for some complex Banach space $G$, an operator $T\in\I(G,F)$ and a mapping $g\in\Hv^\infty(U,G)$.

(P1): Therefore, for all $x\in U$, we have 
$$
v(x)\left\|f(x)\right\|=v(x)\left\|T(g(x))\right\|\leq v(x)\left\|T\right\|\left\|g(x)\right\|\leq \left\|T\right\|_\I\left\|g\right\|_v,
$$
and so $f\in\Hv^\infty(U,F)$ with $\left\|f\right\|_v\leq\left\|T\right\|_\I\left\|g\right\|_v$. Taking the infimum over all such factorizations of $f$ yields $\left\|f\right\|_v\leq\|f\|_{\I\circ\Hv^\infty}$. This implies that $f=0$ if $\left\|f\right\|_{\I\circ\Hv^\infty}=0$.

If $\lambda\in\C$, then $\lambda f=\lambda T\circ g$, hence $\lambda f\in\I\circ\Hv^\infty(U,F)$ with $\|\lambda f\|_{\I\circ\Hv^\infty}\leq\left\|\lambda T\right\|_\I\left\|g\right\|_v=\left|\lambda\right|\left\|T\right\|_\I\left\|g\right\|_v$ and taking the infimum over all the factorizations of $f$ yields $\|\lambda f\|_{\I\circ\Hv^\infty}\leq\left|\lambda\right|\|f\|_{\I\circ\Hv^\infty}$. Conversely, if $\lambda\neq 0$, this implies that $\|f\|_{\I\circ\Hv^\infty}=\|\lambda^{-1}(\lambda f))\|_{\I\circ\Hv^\infty}\leq\left|\lambda\right|^{-1}\|\lambda f\|_{\I\circ\Hv^\infty}$ and thus $\left|\lambda\right|
\|f\|_{\I\circ\Hv^\infty}\leq\|\lambda f\|_{\I\circ\Hv^\infty}$, while if $\lambda=0$, it is clear that $\|\lambda f\|_{\I\circ\Hv^\infty}=0=\left|\lambda\right|\|f\|_{\I\circ\Hv^\infty}$. 

If $f_1,f_2\in\I\circ\Hv^\infty(U,F)$, given $\varepsilon>0$, for each $i=1,2$ we can find a complex Banach space $G_i$, an operator $T_i\in\I(G_i,F)$ and a mapping $g_i\in\Hv^\infty(U,G_i)$ with $\left\|g_i\right\|_v=1$ and $\left\|T_i\right\|_\I\leq\|f_i\|_{\I\circ\Hv^\infty}+\varepsilon/2$ such that $f_i=T_i\circ g_i$. Consider the Banach space $G=G_1\oplus_\infty G_2$ and define the mappings $T\colon G\to F$ and $g\colon U\to G$ by $T(y_1,y_2)=T_1(y_1)+T_2(y_2)$ for all $(y_1,y_2)\in G$ and $g(x)=(g_1(x),g_2(x))$ for all $x\in U$, respectively. An easy calculation shows that $T\in\I(G,F)$ with $\left\|T\right\|_\I\leq\left\|T_1\right\|_\I+\left\|T_2\right\|_\I$ and $g\in\Hv^\infty(U,G)$ with $\left\|g\right\|_v\leq 1$. Clearly, $T\circ g=f_1+f_2$, and so $f_1+f_2\in \I\circ\Hv^\infty(U,F)$ with 
$$
\|f_1+f_2\|_{\I\circ\Hv^\infty}\leq\left\|T\right\|_\I\left\|g\right\|_v\leq\left\|T_1\right\|_\I+\left\|T_2\right\|_\I\leq\|f_1\|_{\I\circ\Hv^\infty}+\|f_2\|_{\I\circ\Hv^\infty}+\varepsilon .
$$
The arbitrariness of $\varepsilon>0$ gives $\|f_1+f_2\|_{\I\circ\Hv^\infty}\leq\|f_1\|_{\I\circ\Hv^\infty}+\|f_2\|_{\I\circ\Hv^\infty}$. So we have proved that $(\I\circ\Hv^\infty(U,F),\|\cdot\|_{\I\circ\Hv^\infty})$ is a normed space.

(P2): Take $h$ and $y$ as in (P2) of Definition \ref{def-weighted holomorphic ideal}, note that $h\cdot y=M_y\circ h$, where $M_y\in\F(\mathbb{C},F)\subseteq\I(\mathbb{C},F)$ is the operator defined by $M_y(\lambda)=\lambda y$ for all $\lambda\in\mathbb{C}$, and so $h\cdot y\in\I\circ\Hv^\infty(U,F)$ with $\|h\cdot y\|_{\I\circ\Hv^\infty}\leq\left\|M_y\right\|_\I\left\|h\right\|_v=\left\|M_y\right\|\left\|h\right\|_v=\left\|y\right\|\left\|h\right\|_v$ and, conversely, $\left\|h\right\|_v\left\|y\right\|=\left\|h\cdot y\right\|_v\leq\|h\cdot y\|_{\I\circ\Hv^\infty}$ by the inequality in (P1). 

(P3): Take $f\in\I\circ\Hv^\infty(U,F)$, and $h$ and $S$ as in (P3) of Definition \ref{def-weighted holomorphic ideal}. Hence $S\circ f\circ h=(S\circ T)\circ (g\circ h)$, where $S\circ T\in\I(G,G_0)$ by the ideal property of $\I$, and $g\circ h\in\H^\infty_v(V,G)$ with $\left\|g\circ h\right\|_v\leq\left\|g\right\|_vc_v(h)$ since 
$$
v(x)\left\|(g\circ h)(x)\right\|=\frac{v(x)}{v(h(x)}v(h(x))\left\|g(h(x))\right\|\leq c_v(h)\left\|g\right\|_v
$$
for all $x\in V$. Therefore $S\circ f\circ h\in \I\circ \Hv^\infty(V,G_0)$ with $\|S\circ f\circ h\|_{\I\circ\Hv^\infty}\leq \left\|S\circ T\right\|_\I\left\|g\circ h\right\|_v\leq\left\|S\right\|\left\|T\right\|_\I\left\|g\right\|_v c_v(h)$. Taking the infimum over all the factorizations of $f$ gives $\left\|S\circ f\circ h\right\|_{\I\circ\Hv^\infty}\leq \left\|S\right\|\left\|f\right\|_{\I\circ\Hv^\infty}c_v(h)$.
\end{proof}


\subsection{Linearization}

We now study the linearization of the composition ideal $\I\circ\Hv^\infty$.

\begin{theorem}\label{ideal}
Let $[\I,\left\|\cdot\right\|_\I]$ be a normed operator ideal and $f\in\Hv^\infty(U,F)$. The following conditions are equivalent: 
\begin{enumerate}
\item $f$ belongs to $\I\circ\Hv^{\infty}(U,F)$.
\item $T_f$ belongs to $\I(\G_v^\infty(U),F)$. 
\end{enumerate}
In this case, $\left\|f\right\|_{\I\circ\Hv^\infty}=||T_f||_\I$. 
Further, the mapping $f\mapsto T_f$ is an isometric isomorphism from $(\I\circ\Hv^\infty(U,F),\left\|\cdot\right\|_{\I\circ\Hv^\infty})$ onto $(\I(\G_v^\infty(U),F),\left\|\cdot\right\|_\I)$.
\end{theorem}

\begin{proof}
$(i)\Rightarrow(ii)$: If $f\in\I\circ\Hv^{\infty}(U,F)$, one can write $f=T\circ g$ for some complex Banach space $G$, $T\in\I(G,F)$ and $g\in\Hv^{\infty}(U,G)$. Since $f=(T\circ T_g)\circ\Delta_v$ and $T\circ T_g\in\L(\Gv^\infty(U),F)$, it follows that $T_f=T\circ T_g$, and thus $T_f\in\I(\Gv^\infty(U),F)$ by the ideal property of $\I$. Further,  
$$
\left\|T_f\right\|_\I=\left\|T\circ T_g\right\|_\I\leq\left\|T\right\|_\I\left\|T_g\right\|=\left\|T\right\|_\I\left\|g\right\|_v, 
$$
and taking the infimum over all factorizations of $f$ as above yields $||T_f||_\I\leq\left\|f\right\|_{\I\circ\Hv^{\infty}}$. 

$(ii)\Rightarrow(i)$: Assume that $T_f\in\I(\Gv^\infty(U),F)$. Since $f=T_f\circ\Delta_v$ and $\Delta_v\in\H^\infty_v(U,\Gv^\infty(U))$, it follows that $f\in\I\circ\Hv^{\infty}(U,F)$ with 
$$
\left\|f\right\|_{\I\circ\Hv^{\infty}}\leq\left\|T_f\right\|_\I\left\|\Delta_v\right\|_v\leq\left\|T_f\right\|_\I. 
$$

For the last assertion, it is sufficient to show that the map $f\mapsto T_f$ in the statement is surjective. For it, 
take $T\in\I(\G_v^\infty(U),F)$, hence $T=T_f$ for some $f\in\H_v^\infty(U,F)$, and then $f\in\I\circ\Hv^{\infty}(U,F)$ by proved above. 
\end{proof}

The last assertion in the statement of Theorem \ref{ideal} gives the following result.

\begin{corollary}
If $[\I,\left\|\cdot\right\|_\I]$ is a Banach operator ideal, then $[\I\circ\Hv^{\infty}, \left\|\cdot\right\|_{\I\circ\Hv^\infty}]$ is a Banach weighted holomorphic ideal.$\hfill\qed$
\end{corollary}

We analyse the relationship of a map $f\in\H_{v\I}^\infty(U,F)$ with its linearization $T_f\in\I(\Gv^\infty(U),F)$. 

\begin{theorem}\label{linear}
Let $p\in [1,\infty]$ and $f\in\Hv^\infty(U,F)$. For the Banach operator ideal $[\I,\left\|\cdot\right\|_\I]$ with $\I=\K_p,\K_{wp},\K_{up}$, the following assertions are equivalent:
\begin{enumerate}
\item $f$ belongs to $\H_{v\I}^\infty(U,F)$.
\item $T_f$ belongs to $\I(\Gv^\infty(U),F)$.
\end{enumerate}
In this case, $\left\|f\right\|_{\H_{v\I}^\infty}=\left\|T_f\right\|_I$. Furthermore, the correspondence $f\mapsto T_f$ is an isometric isomorphism from $(\H_{v\I}^\infty(U,F),\left\|\cdot\right\|_{\H_{v\I}^\infty})$ onto $(\I(\Gv^\infty(U),F),\left\|\cdot\right\|_I)$. As a consequence, $(\H_{v\I}^\infty(U,F),\left\|\cdot\right\|_{\H_{v\I}^\infty})$ is a Banach space.
\end{theorem}

\begin{proof}
We know the following relations:
\begin{align*}
(vf)(U)=T_f(\At_{\G^\infty_v(U)})&\subseteq T_f(B_{\Gv^\infty(U)})=T_f(\overline{\abco}(\At_{\G^\infty_v(U)}))\\
						 &\subseteq \overline{\abco}(T_f(\At_{\G^\infty_v(U)}))=\overline{\abco}((vf)(U)).
\end{align*}
Assume $\I=\K_p,\K_{wp},\K_{up}$. 

$(i)\Rightarrow (ii)$: If $f\in\H_{v\I}^\infty(U,F)$, then $(vf)(U)$ has the $\I$-property. It follows that $\overline{\abco}((vf)(U))$ enjoys the $\I$-property with $m_\I(\overline{\abco}((vf)(U)))=m_\I((vf)(U))$ (see \cite{Kim-14,LasTur-12}). Now, the second inclusion above implies that $T_f(B_{\Gv^\infty(U)})$ has the $\I$-property, that is, $T_f\in\I(\Gv^\infty(U),F)$, with 
$$
\left\|T_f\right\|_\I=m_{\I}(T_f(B_{\Gv^\infty(U)}))\leq m_\I(\overline{\abco}((vf)(U)))=m_\I((vf)(U))=\left\|f\right\|_{\H_{v\I}^\infty},
$$

$(ii)\Rightarrow (i)$: If $T_f\in\I(\Gv^\infty(U),F)$, then $T_f(B_{\Gv^\infty(U)})$ has the $\I$-property. The first inclusion above yields that $(vf)(U)$ has also this property, that is, $f\in\H_{v\I}^\infty(U,F)$, with 
$$
\left\|f\right\|_{\H_{v\I}^\infty}=m_\I((vf)(U))\leq m_\I(T_f(B_{\Gv^\infty(U)}))=\left\|T_f\right\|_\I.
$$
The other assertions of the statement follow immediately.
\end{proof}

In view of Theorem \ref{linear}, Theorem \ref{ideal} and Proposition \ref{ideal comp} yield the following. 

\begin{corollary}\label{proposition: Banach ideal}
$[\H_{v\I}^\infty,\left\|\cdot\right\|_{\H_{v\I}^\infty}]=[\I\circ\H_v^\infty,\left\|\cdot\right\|_{\I\circ\H_v^\infty}]$ and, in particular, $[\H_{v\I}^\infty,\left\|\cdot\right\|_{\H_{v\I}^\infty}]$ is a Banach weighted holomorphic ideal for $\I=\K_p,\K_{wp},\K_{up}$. $\hfill\qed$
\end{corollary}


\subsection{Factorization}

We have justly proved that the weighted holomorphic mappings of the space $\H_{v\I}^\infty$ for $\I=\K_p,\K_{wp},\K_{up}$ can be factorized in terms of an operator $T$ of the corresponding Banach operator ideal $\I$ and a mapping $g$ of the space $\H_v^\infty$.

Using Theorem \ref{linear}, we can improve this factorization by getting $g$ to belong to the space $\H_{v\I}^\infty$. To achieve this goal, we will apply some known factorizations \cite{SinKar-02,Kim-20,Kim-17} of the members of the Banach operator ideal $\I=\K_p,\K_{wp},\K_{up}$. 

\begin{corollary}\label{fact}
Let $p\in [1,\infty)$ and $f\in\Hv^\infty(U,F)$. For the Banach operator ideal $[\I,\left\|\cdot\right\|_\I]$ with $\I=\K_p,\K_{wp},\K_{up}$, the following conditions are equivalent:
\begin{enumerate}
	\item $f\in\H_{v\I}^\infty(U,F)$.
	\item There exist a quotient space $G$ of $\ell_{p^*}$ (of $c_0$ if $p=1$), an operator $S\in\I(G,F)$ and a mapping $g\in\H_{v\mathcal{A}}^\infty(U,G)$ such that $f=S\circ g$, where $\mathcal{A}=\K_{wp}$ if $\I=\K_{p}$, and $\mathcal{A}=\I$ if $\I=\K_{wp},\K_{up}$.
\end{enumerate}
In this case, $\left\|f\right\|_{\H_{v\I}^\infty}=\inf\left\{\left\|S\right\|_\I\left\|g\right\|_{\H_{v\mathcal{A}}^\infty}\right\}$, where the infimum is taken over all such factorizations.
\end{corollary}

\begin{proof}
$(i)\Rightarrow (ii)$: Assume that $f\in\H_{v\I}^\infty(U,F)$. Then $T_f\in\I(\Gv^\infty(U),F)$ with $\left\|f\right\|_{\H_{v\I}^\infty}=\left\|T_f\right\|_I$ by Theorem \ref{linear}. By \cite[Theorem 3.2]{SinKar-02} for $\I=\K_p$, \cite[Proposition 2.4]{Kim-20} for $\I=\K_{wp}$, and \cite[Theorem 2.2]{Kim-17} for $\I=\K_{up}$, given $\varepsilon>0$, there exist a quotient space $G$ of $\ell_{p^*}$ (of $c_0$ if $p=1$), an operator $S\in\I(G,F)$ and a mapping $T\in\mathcal{A}(\Gv^\infty(U),G)$ such that $T_f=S\circ T$ and $\left\|S\right\|_\I\left\|T\right\|_{\mathcal{A}}\leq \left\|T_f\right\|_\I+\varepsilon$. By Theorem \ref{linear} again, there is a mapping $g\in\H_{v\mathcal{A}}^\infty(U,G)$ such that $T=T_g$ and $\left\|g\right\|_{\H_{v{\mathcal{A}}}^\infty}=\left\|T\right\|_{\mathcal{A}}$. Hence $f=T_f\circ\Delta_v=S\circ T\circ\Delta_v=S\circ T_g\circ\Delta_v=S\circ g$ and 
$$
\left\|S\right\|_\I\left\|g\right\|_{\H_{v{\mathcal{A}}}^\infty}=\left\|S\right\|_\I\left\|T\right\|_{\mathcal{A}}\leq\left\|T_f\right\|_\I+\varepsilon=\left\|f\right\|_{\H_{v\I}^\infty}+\varepsilon .
$$
We are done, after passing with $\varepsilon$ to zero.

$(ii)\Rightarrow (i)$: Assume that $f=S\circ g$ with $S$ and $g$ being as in (ii). Since $(vf)(U)=S((vg)(U))$ and $(vg)(U)$ is norm bounded, it follows that $(vf)(U)$ has the $\I$-property, that is, $f\in\H_{v\I}^\infty(U,F)$. Also, $T_f=S\circ T_g$. Moreover, by Theorem \ref{linear}, Theorem \ref{ideal}, Definition \ref{composition ideal}, 
$$
\left\|f\right\|_{\H_{v\I}^\infty}=\|T_f\|_\I=\|S\circ T_g\|_\I\leq \|S\|_{\I}\|T_g\| \leq \|S\|_{\I}\|T_g\|_{\mathcal{A}}  =\left\|S\right\|_\I\left\|g\right\|_{\H_{v\mathcal{A}}^\infty},
$$
and thus $\left\|f\right\|_{\H_{v\I}^\infty}\leq\inf\left\{\left\|S\right\|_\I\left\|g\right\|_{\H_{v\mathcal{A}}^\infty}\right\}$ by taking the infimum over all factorizations of $f$ as above.
\end{proof}

We now present both factorizations for $p$-compact weighted holomorphic maps which should be compared to those established in \cite[Proposition 2.9]{GalLasTur-12} and \cite[Theorem 2.3]{Kim-17} for $p$-compact linear operators.

\begin{corollary}\label{corespcoc}
Let $p\in [1,\infty)$ and $f\in\Hv^\infty(U,F)$. The following conditions are equivalent: 
\begin{enumerate}
\item $f\in\H_{v\K_p}^\infty(U,F)$. 
\item There exist a closed subspace $M$ in $\ell_{p^{*}}$ ($c_0$ instead of $\ell_{p^{*}}$ if $p=1$), a separable Banach space $G$, an operator $T$ in $\K_p(\ell_{p^{*}}/M,G)$, a mapping $g$ in $\H_{v\K}^\infty(U,\ell_{p^{*}}/M)$ and an operator $S$ in $\K(G,F)$ such that $f=S\circ T\circ g$. 
\end{enumerate}
In this case, $\left\|f\right\|_{\H_{v\K_p}^\infty}=\inf\{||S||\left\|T\right\|_{\K_p}\|g\|_v\}$, where the infimum is extended over all factorizations of $f$ as above. 
\end{corollary}

\begin{proof}
$(i)\Rightarrow (ii)$: Suppose that $f\in\H^\infty_{v\K_p}(U,F)$. By Theorem \ref{linear}, $T_f\in\K_p(\G_v^\infty(U),F)$ with $\left\|T_f\right\|_{\K_p}=\left\|f\right\|_{\H_{v\K_p}^\infty}$. Applying \cite[Proposition 2.9]{GalLasTur-12}, for each $\varepsilon>0$, there exist a closed subspace $M\subseteq\ell_{p^*}$ ($c_0$ instead of $\ell_{p^*}$ if $p=1$), a separable Banach space $G$, an operator $T\in\K_p(\ell_{p^*}/M,G)$, an operator $S\in\K(G,F)$ and an operator $R\in\K(\Gv^\infty(U),\ell_{p^*}/M)$ such that $T_f=S\circ T\circ R$ with $\left\|S\right\|\left\|T\right\|_{\K_p}\left\|R\right\|\leq \left\|T_f\right\|_{\K_p}+\varepsilon$. Moreover, $R=T_g$ with $\left\|g\right\|_v=\left\|R\right\|$ for some $g\in\H_{v\K}^\infty(U,\ell_{p^{*}}/M)$. Thus, we obtain  
$$
f=T_f\circ\Delta_v=S\circ T\circ R\circ\Delta_v=S\circ T\circ T_g\circ\Delta_v=S\circ T\circ g,
$$
with 
$$
\left\|S\right\|\left\|T\right\|_{\K_p}\left\|g\right\|_v=\left\|S\right\|\left\|T\right\|_{\K_p}\left\|R\right\|\leq \left\|T_f\right\|_{\K_p}+\varepsilon=\left\|f\right\|_{\H_{v\K_p}^\infty}+\varepsilon.
$$
Since $\varepsilon$ was arbitrary, we deduce that $\left\|S\right\|\left\|T\right\|_{\K_p}\left\|g\right\|_v\leq \left\|f\right\|_{\H_{v\K_p}^\infty}$.

$(ii)\Rightarrow (i)$: Assume that $f=S\circ T\circ g$ is a factorization as in $(ii)$. Since $S\circ T\in\K_p(\ell_{p^*}/M,F)$ by the ideal property of $\K_p$, Corollary \ref{proposition: Banach ideal} yields that $f\in\H^\infty_{v\K_p}(U,F)$ with 
$$
\left\|f\right\|_{\H_{v\K_p}^\infty}\leq \left\|S\circ T\right\|_{\K_p}\left\|g\right\|_v\leq \left\|S\right\|\left\|T\right\|_{\K_p}\left\|g\right\|_v,
$$ 
and taking infimum over all such factorizations of $f$, we have $\left\|f\right\|_{\H_{v\K_p}^\infty}\leq\inf\{\left\|S\right\|\left\|T\right\|_{\K_p}\left\|g\right\|_v\}$. 
\end{proof}

The next result --which can be proven similarly to Corollary \ref{corespcoc} using \cite[Theorem 2.3]{Kim-17}-- is a factorization theorem stronger than Corollary \ref{fact} for $\I=\K_p$ and Corollary \ref{corespcoc}.

\begin{corollary}\label{corespcocKp}
	Let $p\in [1,\infty)$ and $f\in\Hv^\infty(U,F)$. The following conditions are equivalent: 
	\begin{enumerate}
		\item $f\in\H_{v\K_p}^\infty(U,F)$. 
		\item There exist a quotient space $Z$ of $\ell_{p^{*}}$ ($c_0$ instead of $\ell_{p^{*}}$ if $p=1$), an operator $S$ in $\K_p(Z,F)$, a mapping $g$ in $\H_{v\K_{up}}^\infty(U,Z)$ such that $f=S\circ g$. 
	\end{enumerate}
	In this case, $\left\|f\right\|_{\H_{v\K_{up}}^\infty}=\inf\{\left\|S\right\|_{\K_p}\|g\|_{\H_{v\K_{up}}^\infty}\}$, where the infimum is taken over all possible $S'$s and $g$'s. $\hfill\qed$
\end{corollary}

The celebrated Davis--Figiel--Johnson--Pe\l czynski theorem \cite{DFJP-74} asserts that any weakly compact linear operator factors through a reflexive Banach space. Since $[\K_{wp},\left\|\cdot\right\|_{\K_{wp}}]\leq [\K_w,\left\|\cdot\right\|]$ whenever $p\in (1,\infty)$, the cited theorem and Corollary \ref{fact} yields immediately the following. 

\begin{corollary}
Let $p\in (1,\infty)$ and $f\in\Hv^\infty(U,F)$. The following are equivalent:
\begin{enumerate}
	\item $f\in\HvWp^\infty(U,F)$.
	\item There are a quotient space $G$ of $\ell_{p^*}$, a reflexive Banach space $H$, operators $T\in\L(G,H)$ and $Q\in\L(H,F)$ and a mapping $g\in\H_{v\K_{wp}}^{\infty}(U,G)$ such that $f=Q\circ T \circ g$. 
\end{enumerate} 
In this case, we have 
$$
\inf\{\left\|Q\circ T\right\|\left\|g\right\|_{\H_{v\K_{wp}}^\infty}\}\leq\left\|f\right\|_{\H_{v\K_{wp}}^\infty}\leq\inf\{\left\|Q\right\|\left\|T\right\|\left\|g\right\|_{\H_{v\K_{wp}}^\infty}\},
$$
where both infimums are taken over all such factorizations of $f$ as in (ii). $\hfill\Box$
\end{corollary}


\subsection{Transposition}

The Banach operator ideals $\K_p$, $\K_{wp}$ and $\K_{up}$ are associated by duality with the ideals of quasi $p$-nuclear operators, quasi weakly $p$-nuclear operators and quasi unconditionally $p$-nuclear operators, respectively. 

According to \cite{PerPie-69,Kim-20,Kim-14}, for every $p\in [1,\infty)$, an operator $T\in\L(E,F)$ is said to be \textit{quasi $p$-nuclear (resp., quasi weakly $p$-nuclear, quasi unconditionally $p$-nuclear)} if there is a sequence $(x_n^{*})\in\ell_p(E^{*})$ (resp., $(x_n^{*})\in\ell_p^w(E^*)$, $(x_n^{*})\in\ell_p^u(E^*)$) such that $\left\|T(x)\right\|\leq \left\|(x_n^{*}(x))\right\|_p$ for every $x\in E$.

The space of all quasi $p$-nuclear (resp., quasi weakly $p$-nuclear, quasi unconditionally $p$-nuclear) operators from $E$ into $F$ is denoted by $\QN_p(E,F)$ (resp., $\QN_{wp}(E,F)$, $\QN_{up}(E,F)$). 

It is known that $[\QN_p,\left\|\cdot\right\|_{\QN_p}]$, equipped with the norm 
$$
\left\|T\right\|_{\QN_p}=\inf\left\{\|(x_n^{*})\|_p\colon ||T(x)||\leq \left\|(x_n^{*}(x))\right\|_p,\; \forall x\in E\right\},
$$
is a Banach operator ideal. Similarly, $[\QN_{wp},\left\|\cdot\right\|_{\QN_{wp}}]$ and $[\QN_{up},\left\|\cdot\right\|_{\QN_{wp}}]$, endowed with the norm 
$$
\left\|T\right\|_{\QN_{wp}}=\inf\left\{\|(x_n^{*})\|^w_p\colon ||T(x)||\leq \left\|(x_n^{*}(x))\right\|_p,\; \forall x\in E\right\},
$$
are Banach operator ideals. 

By \cite[Proposition 3.8]{DelPinSer-10} and \cite[Corollary 2.7]{GalLasTur-12} for $\I=\K_p$ and $i=p$, and \cite[Theorem 5.6]{Kim-17} for $\I=\K_{up}$ and $i=up$, an operator $T\in\I(E,F)$ if and only if its adjoint $T^*\in\QN_i(F^*,E^*)$, in whose case $\left\|T\right\|_{\I}=\left\|T^*\right\|_{\QN_i}$. 

A weighted holomorphic version of these results can be stated as follows.

\begin{corollary}\label{fact-2}
Let $p\in (1,\infty)$ and $f\in\Hv^\infty(U,F)$. For the Banach operator ideal $[\I,\left\|\cdot\right\|_\I]$ with $\I=\K_p$ (resp., $\K_{up}$) and $i=p$ (resp., $up$), the following conditions are equivalent:
\begin{enumerate}
	\item $f\in\H_{v\I}^\infty(U,F)$.
	\item $f^t\in\QN_i(F^*,\Hv^\infty(U))$. 
\end{enumerate}
In this case, $\left\|f\right\|_{\H_{v\I}^\infty}=\left\|f^t\right\|_{\QN_i}$.
\end{corollary}

\begin{proof}
Applying Theorem \ref{linear} and the results preceding to this corollary, we have
\begin{align*}
f\in\H^\infty_{v\I}(U,F)&\Leftrightarrow T_f\in\I(\G_v^\infty(U),F)\\
                             &\Leftrightarrow (T_f)^*\in\QN_i(F^*,\G_v^\infty(U)^*)\\
                             &\Leftrightarrow f^t=J_v^{-1}\circ(T_f)^*\in\QN_i(F^*,\H^\infty(U)).
\end{align*}
Moreover, $\left\|f\right\|_{\H_{v\I}^\infty}=\left\|T_f\right\|_{\I}=\left\|(T_f)^*\right\|_{\QN_i}=\left\|f^t\right\|_{\QN_i}$. 
\end{proof}

In \cite[Theorem 3.7 (c)]{Kim-19}, Kim proved that the adjoint of a weakly $p$-compact operator $T$ is quasi weakly $p$-nuclear and $\left\|T^*\right\|_{\QN_{wp}}\leq\left\|T\right\|_{\K_{wp}}$. Using this fact, we obtain the following with a proof similar to that of Corollary \ref{fact-2}.

\begin{corollary}
Let $p\in (1,\infty)$. If $f\in\HvWp^\infty(U,F)$, then $f^t\in\QN_{wp}(F^*,\Hv^\infty(U))$ and $\|f^t\|_{\QN_{wp}}\leq\left\|f\right\|_{\H_{v\K_{wp}}^\infty}$. $\hfill\Box$
\end{corollary}

Given $p\in [1,\infty)$, let us recall (see \cite{Pie-80}) that an operator $T\in\L(E,F)$ is \textit{p-summing} if there exists a constant $C\geq 0$ such that
$$
\left(\sum_{i=1}^n||T(x_i)||^p\right)^{\frac{1}{p}}\leq C\sup_{x^{*}\in B_{E^{*}}}\left(\sum_{i=1}^n\left|x^{*}(x_i)\right|^p\right)^{\frac{1}{p}}
$$
for all $n\in\mathbb{N}$ and $x_1,\ldots, x_n\in E$. The infimum of such constants $C$ is denoted by $\pi_p(T)$ and the linear space of all $p$-summing operators from $E$ into $F$ by $\Pi_p(E,F)$. The pair $[\Pi_p,\pi_p]$ is a Banach operator ideal.

Based on \cite[Proposition 3.13]{DelPinSer-10}, we propose a weighted holomorphic variant of this result.

\begin{corollary}\label{prop-25}
Let $f\in\Hv^\infty(U,F)$ and $g\in\H_{v\K}^\infty(U,F^*)$. Assume that $T_f\in\Pi_p(\Gv^\infty(U),F)$ with $p\in [1,\infty)$. Then $f^t\circ g\in\H^\infty_{v\K_p}(U,\Hv^\infty(U))$ with $\left\|f^t\circ g\right\|_{\H_{v\K_p}^\infty}\leq\pi_p(T_f)\left\|g\right\|_v$.  
\end{corollary}

\begin{proof}
Observe that $T_g\in\K(\G_v^\infty(U),F^*)$ with $||T_g||=\left\|g\right\|_v$. Hence $(T_f)^*\circ T_g\in\K_p(\Gv^\infty(U),\Gv^\infty(U)^*)$ with $\left\|(T_f)^*\circ T_g\right\|_{\K_p}\leq \pi_p(T_f)||T_g||$ by \cite[Proposition 3.13]{DelPinSer-10}. From the equality $f^t\circ T_g=J_v^{-1}\circ (T_f)^*\circ T_g$, we infer that $f^t\circ T_g\in\K_p(\Gv^\infty(U),\Hv^\infty(U))$ with $\left\|f^t\circ T_g\right\|_{\K_p}=\left\|(T_f)^*\circ T_g\right\|_{\K_p}$ by the ideal property of $\K_p$. Applying Theorem \ref{linear}, there exists $h\in\H^\infty_{v\K_p}(U,\Hv^\infty(U))$ with $\left\|h\right\|_{\H_{v\K_p}^\infty}=\left\|T_h\right\|_{\K_p}$ such that $f^t\circ T_g=T_h$. Hence $f^t\circ g=h$ and thus $f^t\circ g\in\H^\infty_{v\K_p}(U,\Hv^\infty(U))$ with $\left\|f^t\circ g\right\|_{\H_{v\K_p}^\infty}=\left\|T_h\right\|_{\K_p}=\left\|(T_f)^*\circ T_g\right\|_{\K_p}\leq\pi_p(T_f)\left\|T_g\right\|=\pi_p(T_f)\left\|g\right\|_v$.
\end{proof}

A proof similar to that of Corollary \ref{prop-25} but applying \cite[Proposition 5.4]{SinKar-02} instead of \cite[Proposition 3.13]{DelPinSer-10} provides another result closely related.

\begin{proposition}
Let $f\in\Hv^\infty(U,F)$ and $g\in\HvWp^\infty(U,F^*)$. Assume that $(T_f)^*\in\Pi_p(F^*,\Gv^\infty(U)^*)$ with $p\in [1,\infty)$. Then $f^t\circ g\in\H^\infty_{v\K_p}(U,\Hv^\infty(U))$ with $\left\|f^t\circ g\right\|_{\H_{v\K_p}^\infty}\leq \pi_p((T_f)^*)\left\|g\right\|_{\H_{v\K_{wp}}^\infty}$.  $\hfill\qed$
\end{proposition}

As we have recalled, $p$-compact operators are characterized as those operators $T$ for which $T^*$ factor through a subspace of $\ell_p$. We now obtain a similar factorization for the transpose of a map in $\HvKp^\infty(U,F)$.

\begin{corollary}
Let $p\in [1,\infty)$ and $f\in\Hv^\infty(U,F)$. The following conditions are equivalent:
\begin{enumerate}
	\item $f\in\H_{v\K_p}^\infty(U,F)$. 
	\item There exist a closed subspace $M\subseteq\ell_p$ and $R\in\QN_p(F^{*},M)$, $S\in\L(M,\Hv^\infty(U))$ such that $f^t=S\circ R$. 
\end{enumerate}
\end{corollary}

\begin{proof}
$(i)\Rightarrow (ii)$: If $f\in\H^\infty_{v\K_p}(U,F)$, then $T_f\in\K_p(\Gv^\infty(U),F)$ by Theorem \ref{linear}. By \cite[Proposition 3.10]{DelPinSer-10}, there exist a closed subspace $M\subseteq\ell_{p}$ and operators $R\in\QN_p(F^*,M)$ and $S_0\in\L(M,\Gv^\infty(U)^*)$ such that $(T_f)^*=S_0\circ R$. Taking $S=J_v^{-1}\circ S_0\in\L(M,\Hv^\infty(U))$, we have $f^t=S\circ R$.

$(ii)\Rightarrow (i)$: Assume that $f^t=S\circ R$ being $S$ and $R$ as in the statement. Then $(T_f)^*=J_v\circ f^t=J_v\circ S\circ R$, and thus $T_f\in\K_p(\Gv^\infty(U),F)$ by \cite[Proposition 3.10]{DelPinSer-10}. Hence $f\in\H^\infty_{v\K_p}(U,F)$ by Theorem \ref{linear}
\end{proof}


\subsection{Relations with $p$-approximable type weighted holomorphic mappings}

In similarity with the linear case, it seems natural to introduce the following classes of weighted holomorphic mappings which are related to the preceding ones. The particular case of bounded holomorphic mappings was dealed by Mujica in \cite{Muj-91}.

\begin{definition}
A mapping $f\in\Hv^\infty(U,F)$ is said to \textit{have finite-dimensional $v$-rank} if the linear hull of $(vf)(U)$ is a finite-dimensional subspace of $F$. In that case, we define the \textit{v-$\rank(f)$ of $f$} to be the dimension of this subspace. We denote by $\H_{v\F}^\infty(U,F)$ the space of all weighted holomorphic mappings from $U$ to $F$ that have finite-dimensional $v$-rank. 

A mapping $f\in\Hv^\infty(U,F)$ is said to be \textit{approximable} if there exists a sequence $(f_n)$ in $\H_{v\F}^\infty(U,F)$ such that $\left\|f_n-f\right\|_v\to 0$ as $n\to\infty$.

Given $p\in [1,\infty)$ and $\I=\K_p,\K_{wp},\K_{up}$, a mapping $f\in\H_{v\I}^\infty(U,F)$ is said to be \textit{$p$-approximable (resp., $wp$-approximable, $up$-approximable)} if there exists a sequence $(f_n)$ in $\H_{v\F}^\infty(U,F)$ such that $\left\|f_n-f\right\|_{H_{v\I}^{\infty}}\to 0$ as $n\to\infty$. 

We denote by $\H^\infty_{v\overline{\F}}(U,F)$ (resp., $\H^\infty_{v\overline{\F}^\I}(U,F)$,  $\I=\K_p, \K_{wp}, \K_{up}$) the space of all approximable (resp., $p$-approximable, $wp$-approximable, $up$-approximable) weighted holomorphic mappings from $U$ to $F$.
\end{definition}

We now study the linearization and transposition of these mappings.

\begin{proposition}\label{prop-18:07}
Let $f\in\Hv^\infty(U,F)$. The following are equivalent: 
\begin{enumerate}
	\item $f\in\H_{v\F}^\infty(U,F)$. 
	\item $T_f\in\F(\G_v^\infty(U),F)$.
	\item $f^t\in\F(F^*,\H_v^\infty(U))$.
\end{enumerate}
In this case, $v$-$\rank(f)=\rank(T_f)$. Furthermore, the correspondence $f\mapsto T_f$ is an isometric isomorphism from $(\H_{v\F}^\infty(U,F),\left\|\cdot\right\|_v)$ onto $(\F(\Gv^\infty(U),F),\left\|\cdot\right\|)$.
\end{proposition}

\begin{proof}
$(i)\Leftrightarrow (ii)$ and $v$-$\rank(f)=\rank(T_f)$ are deduced from the equality   
$$
\lin((vf)(U))=\lin(T_f(\At_{\G_v^\infty(U)})=T_f(\lin(\At_{\G_v^\infty(U)}))=T_f(\G_v^\infty(U)).
$$ 
$(ii)\Leftrightarrow (iii)$ follows from 
\begin{align*}
T_f\in\F(\G_v^\infty(U),F)&\Leftrightarrow (T_f)^*\in\F(F^*,\G_v^\infty(U)^*)\\
                          &\Leftrightarrow f^t=J_v^{-1}\circ(T_f)^*\in\F(F^*,\H^\infty(U))
\end{align*}
since $\F$ is a completely symmetric operator ideal \cite[4.4.7]{Pie-80}.  
\end{proof}

In view of Theorem \ref{linear}, Propositions \ref{ideal comp} and \ref{prop-18:07} give the following. 

\begin{corollary}\label{proposition: Banach ideal finito}
$[\H_{v\F}^\infty,\left\|\cdot\right\|_v]=[\F\circ\H_v^\infty,\left\|\cdot\right\|_{\F\circ\H_v^\infty}]$ and, in particular, $[\H_{v\F}^\infty,\left\|\cdot\right\|_v]$ is a normed weighted holomorphic ideal. $\hfill\qed$
\end{corollary}

Corollary \ref{proposition: Banach ideal} shows that $[\H_{v\F}^\infty,\left\|\cdot\right\|_v]\leq [\H_{v\I}^\infty,\left\|\cdot\right\|_{\H_{v\I}^\infty}]$ for $\I= \K_p, \K_{wp}, \K_{up}$. Now, we give the following result.

\begin{corollary}
Let $p\in [1,\infty)$ and $\I= \K_p, \K_{wp}, \K_{up}$. Then $\H^\infty_{v\overline{\F}^\I}(U,F)\subseteq \H_{v\I}^\infty(U,F)$.
\end{corollary}
\begin{proof}
Let $f\in\H^\infty_{v\overline{\F}^\I}(U,F)$. Hence there is a sequence $(f_n)$ in $\H^\infty_{v\F}(U,F)$ such that $\left\|f_n-f\right\|_{\H^\infty_{v\I}}\to 0$ as $n\to\infty$. Since $T_{f_n}\in\F(\Gv^\infty(U),F)$ by Proposition \ref{prop-18:07}; $\F(\Gv^\infty(U),F)\subseteq\I(\Gv^\infty(U),F)$ by \cite[Theorem 4.2]{SinKar-02}, \cite[Theorem 4.1]{SinKar-02} and \cite[Theorem 2.1]{Kim-14} for $\I= \K_p, \K_{wp}, \K_{up}$, respectively; and $\left\|T_{f_n}-T_f\right\|_{\I}=\left\|f_n-f\right\|_{\H_{v\I}^{\infty}}$ for all $n\in\mathbb{N}$ by Theorem \ref{linear}, we deduce that $T_f\in \overline{\F(\Gv^\infty(U),F)}^{\I}$. Hence $T_f\in\I(\Gv^\infty(U),F)$ by the same results of \cite{SinKar-02} and \cite{Kim-14} cited above, and so $f\in\H^\infty_{v\I}(U,F)$ by Theorem \ref{linear}.
\end{proof}

We can also easily prove the following result using Proposition \ref{prop-18:07} and the fact that the class of approximable linear operators is also a completely symmetric operator ideal \cite[4.4.7]{Pie-80}.

\begin{proposition}\label{prop-18:08}
Let $f\in\Hv^\infty(U,F)$. The following are equivalent: 
\begin{enumerate}
	\item $f\in\H_{v\overline{\F}}^\infty(U,F)$. 
	\item $T_f\in\overline{\F}(\G_v^\infty(U),F)$.
	\item $f^t\in\overline{\F}(F^*,\H_v^\infty(U))$.
\end{enumerate}
Furthermore, the correspondence $f\mapsto T_f$ is an isometric isomorphism from $(\H_{v\overline{\F}}^\infty(U,F),\left\|\cdot\right\|_v)$ onto $(\overline{\F}(\Gv^\infty(U),F),\left\|\cdot\right\|)$.$\hfill\qed$
\end{proposition}

Now, Theorem \ref{linear}, Propositions \ref{ideal comp} and \ref{prop-18:08} provide the following. 

\begin{corollary}\label{proposition: Banach ideal finito1}
$[\H_{v\overline{\F}}^\infty,\left\|\cdot\right\|_v]=[\overline{\F}\circ\H_v^\infty,\left\|\cdot\right\|_{\overline{\F}\circ\H_v^\infty}]$ and, in particular, $[\H_{v\overline{\F}}^\infty,\left\|\cdot\right\|_v]$ is a Banach weighted holomorphic ideal. $\hfill\qed$
\end{corollary}

Let us recall that a Banach space $E$ is said to have the \textit{approximation property} if given a compact set $K\subseteq E$ and $\varepsilon>0$, there is an operator $T\in\F(E,E)$ such that $\left\|T(x)-x\right\|<\varepsilon$ for every $x\in K$. On the approximation property in holomorphic function spaces, we refer to Mujica \cite{Muj-91}.

Grothendieck \cite{g} proved that a dual Banach space $E^*$ has the approximation property if and only if given a Banach space $F$, an operator $S\in\K(E,F)$ and $\varepsilon>0$, there is an operator $T\in\F(E,F)$ such that $\left\|T-S\right\|<\varepsilon$. Since $\Hv^\infty(U)$ is a dual space, combining the Grothendieck's result with Proposition \ref{new}, Theorem \ref{linear} and Proposition \ref{prop-18:08}, we next give a necessary and sufficient condition for the space $\Hv^\infty(U)$ have the approximation property.

\begin{corollary}
$\Hv^\infty(U)$ has the approximation property if and only if $\H_{v\K}^\infty(U,F)\subseteq\H_{v\overline{\F}}^\infty(U,F)$ for each complex Banach space $F$. $\hfill\qed$
\end{corollary}

\subsection{Relations with right $p$-nuclear weighted holomorphic mappings}

Let us recall (see \cite{Per-69}) that an operator $T\in\L(E,F)$ is said to be \textit{right $p$-nuclear} if there are sequences $(x_n^{*})\in\ell_{p^*}^{w}(E^{*})$ and $(y_n)\in\ell_p(F)$ such that
$$
T(x)=\sum_{n=1}^\infty x_n^{*}(x)y_n\qquad (x\in E)
$$
and the series converges in $\L(E,F)$. The set of such operators, denoted $\Nu^p(E,F)$, is a Banach space with the norm
$$
\left\|T\right\|_{\Nu^p}=\inf\left\{\|(x_n^{*})\|_{p^{*}}^w\|y_n\|_p\right\},
$$
where the infimum is taken over all representations of $T$ as above. 

A weighted holomorphic variant of this class of operators can be introduced as follows.

\begin{definition}
Given $p\in [1,\infty)$, a mapping $f\in\H_{v}^{\infty}(U,F)$ is said to be \textit{right $p$-nuclear} if there exist sequences $(g_n)$ in $\ell_{p^{*}}^{w}(\Hv^\infty(U))$ and $(y_n)$ in $\ell_p(F)$ such that $f=\sum_{n=1}^\infty g_n\cdot y_n$ in $\left(\Hv^\infty(U,F), \|\cdot\|_v\right)$. We set 
$$
\left\|f\right\|_{\H^\infty_{v\Nu_p}}=\inf\left\{\|(g_n)\|_{p^{*}}^w\|(y_n)\|_p\right\},
$$
with the infimum taken over all right $p$-nuclear weighted holomorphic representations of $f$ as above. Let $\H^\infty_{v\Nu^p}(U,F)$ denote the set of all right $p$-nuclear weighted holomorphic mappings from $U$ into $F$. 
\end{definition}

We now linearize right $p$-nuclear weighted holomorphic mappings. 

\begin{proposition} \label{teorpnu}
Let $p\in [1,\infty)$ and $f\in\Hv^\infty(U,F)$. The following conditions are equivalent:
\begin{enumerate}
	\item $f\in\H^\infty_{v\Nu^p}(U,F)$. 
	\item $T_f\in\Nu^p(\G_v^\infty(U),F)$. 
\end{enumerate}
In this case, we have $\left\|f\right\|_{\H^\infty_{v\Nu^p}}=\left\|T_f\right\|_{\Nu^p}$. Furthermore, the mapping $f\mapsto T_f$ is an isometric isomorphism from $\left(\H^\infty_{v\Nu^p}(U,F), \left\|\cdot\right\|_{\H^\infty_{v\Nu^p}}\right)$ onto $\left(\Nu^p(\G_v^\infty(U),F),\left\|\cdot\right\|_{\Nu^p}\right)$.
\end{proposition}

\begin{proof}
$(i)\Rightarrow (ii)$: Assume that $f\in\H^\infty_{v\Nu^p}(U,F)$ and let $\sum_{n= 1}^{\infty}g_n\cdot y_n$ be a right $p$-nuclear weighted holomorphic representation of $f$. Take the unique operator $T_f\in\L(\Gv^\infty(U),F)$ such that $T_f\circ\Delta_v=f$. Similarly, for each $n\in\mathbb{N}$, take $T_{g_n}\in\Gv^\infty(U)^*$ with $||T_{g_n}||=\left\|g_n\right\|_v$ such that $T_{g_n}\circ\Delta_v=g_n$. Notice that $\sum_{n=1}^{\infty}T_{g_n}\cdot y_n\in\L(\Gv^\infty(U),F)$ since 
$$
\sum_{k=1}^m\left\|T_{g_k}\cdot y_k\right\|=\sum_{k=1}^m\left\|T_{g_k}\right\|\left\|y_k\right\|=\sum_{k=1}^m\left\|g_k\right\|_v\left\|y_k\right\|\leq\left\|(g_n)\right\|^w_{p^*}\left\|(y_n)\right\|_{p},
$$
for all $m\in\mathbb{N}$. We can write
$$
f=\sum_{n=1}^\infty g_n\cdot y_n=\sum_{n=1}^\infty (T_{g_n}\circ\Delta_v)\cdot y_n=\left(\sum_{n=1}^\infty T_{g_n}\cdot y_n\right)\circ \Delta_v,
$$
in $(\Hv^\infty(U,F),\left\|\cdot\right\|_v)$. Hence $T_f=\sum_{n=1}^\infty T_{g_n}\cdot y_n$, where $(T_{g_n})\in\ell^w_{p^*}(\Gv^\infty(U)^*)$ with $\left\|(T_{g_n})\right\|^w_{p^*}\leq\left\|(g_n)\right\|^w_{p^*}$. Therefore $T_f\in\Nu^p(\Gv^\infty(U),F)$ with $\left\|T_f\right\|_{\Nu^p}\leq\left\|(g_n)\right\|^w_{p^*} \left\|(y_n)\right\|_p$. Taking infimum over all right $p$-nuclear weighted holomorphic representations of $f$, we deduce that $\left\|T_f\right\|_{\Nu^p}\leq\left\|f\right\|_{\H^\infty_{v\Nu^p}}$.

$(ii)\Rightarrow (i)$: Suppose that $T_f\in\Nu^p(\Gv^\infty(U),F)$ and let $\sum_{n= 1}^{\infty}\phi_n\cdot y_n$ be a right $p$-nuclear representation of $T_f$. For each $n\in\mathbb{N}$, there is a $g_n\in\Hv^\infty(U)$ such that $J_v(g_n)=\phi_n$ with $\left\|g_n\right\|_v=||\phi_n||$. We have 
\begin{align*}
v(x)\left\|\left(f-\sum_{k=1}^ng_k\cdot y_k\right)(x)\right\|&=v(x)\left\|f(x)-\sum_{k=1}^ng_k(x)y_k\right\|\\
&=v(x)\left\|T_f(\Delta_v(x))-\sum_{k=1}^n J_v(g_k)(\Delta_v(x))y_k\right\|\\
&=v(x)\left\|\left(T_f-\sum_{k=1}^n\phi_k\cdot y_k\right)(\Delta_v(x))\right\|\\
&\leq \left\|T_f-\sum_{k=1}^n\phi_k\cdot y_k\right\|v(x)\left\|\Delta_v(x)\right\|\\
&\leq\left\|T_f-\sum_{k=1}^n\phi_k\cdot y_k\right\|
\end{align*}
for all $x\in U$ and $n\in\mathbb{N}$. Taking supremum over all $x\in U$, we obtain
$$
\left\|f-\sum_{k=1}^ng_k\cdot y_k\right\|_v\leq\left\|T_f-\sum_{k=1}^n\phi_k\cdot y_k\right\| ,
$$
for all $n\in\mathbb{N}$. Hence $f=\sum_{n=1}^\infty g_n\cdot y_n$ in $(\Hv^\infty(U,F),\left\|\cdot\right\|_v)$, where $(g_n)\in\ell^w_{p^*}(\Hv^\infty(U))$ with $\left\|(g_n)\right\|^w_{p^*}\leq\left\|(\phi_n)\right\|^w_{p^*}$. So $f\in\H^\infty_{v\Nu^p}(U,F)$ with $\left\|f\right\|_{\H^\infty_{v\Nu^p}}\leq\left\|(\phi_n)\right\|^w_{p^*}\left\|(y_n)\right\|_p$, and therefore $\left\|f\right\|_{\H^\infty_{v\Nu^p}}\leq\left\|T_f\right\|_{\Nu^p}$.

\end{proof}

Combining Proposition \ref{teorpnu} with Theorem \ref{ideal} and Proposition \ref{ideal comp}, we derive the following.

\begin{corollary}\label{proposition: Banach ideal-8}
$[\H_{v\Nu^p}^\infty,\left\|\cdot\right\|_{\H_{v\Nu^p}^\infty}]=[\Nu^p\circ\H_v^\infty,\left\|\cdot\right\|_{\Nu^p\circ\H_v^\infty}]$ and, in particular, $[\H_{v\Nu^p}^\infty,\left\|\cdot\right\|_{\H_{v\Nu^p}^\infty}]$ is a Banach weighted holomorphic ideal for any $p\in [1,\infty)$. $\hfill\qed$
\end{corollary}

We now show that every right $p$-nuclear weighted holomorphic mapping has a relatively $p$-compact $v$-range.

\begin{corollary}
$[\H_{v\Nu^p}^\infty,\left\|\cdot\right\|_{\H_{v\Nu^p}^\infty}]\leq[\H_{v\K_p}^\infty,\left\|\cdot\right\|_{\H_{v\K_p}^\infty}]$ for any $p\in [1,\infty)$.
\end{corollary}

\begin{proof}
Let  $f\in\H^\infty_{v\Nu^p}(U,F)$. By Proposition \ref{teorpnu}, we have $T_f\in\Nu^p(\Gv^\infty(U),F)$ with $\left\|T_f\right\|_{\Nu^p}=\left\|f\right\|_{\H^\infty_{v\Nu^p}}$. It follows that $T_f\in\K_p(\Gv^\infty(U),F)$ with $\left\|T_f\right\|_{\K_p}\leq\left\|T_f\right\|_{\Nu^p}$ (see \cite[p. 295]{DelPinSer-10}). Hence $f\in\H^\infty_{v\K_p}(U,F)$ with $\left\|f\right\|_{\H^\infty_{v\K_p}}\leq\left\|f\right\|_{\H^\infty_{v\Nu^p}}$ by Theorem \ref{linear}.
\end{proof}

\textbf{Conflict of interest}. The authors have no relevant financial or non-financial interests to disclose.


\begin{thebibliography}{1}
\bibitem{AchDahTur-19} D. Achour, E. Dahia and P. Turco, Lipschitz $p$-compact mappings, Monatsh. Math. {\bf 189} (2019), no. 4, 595--609.

\bibitem{AroBotPelRue-10} R. Aron, G. Botelho, D. Pellegrino ad P. Rueda, Holomorphic mappings associated to composition ideals of polynomials, Atti Accad. Naz. Lincei Rend. Lincei Mat. Appl. {\bf 21} (2010), no. 3, 261--274. 

\bibitem{BieSum-93} K. D. Bierstedt and W. H. Summers, Biduals of weighted Banach spaces of analytic functions, J. Austral. Math. Soc. Ser. A \textbf{54} (1993), no. 1, 70--79. 

\bibitem{BieSum-93-2} K. D. Bierstedt, J. Bonet and A. Galbis, Weighted spaces of holomorphic functions on balanced domains, Michigan Math. J. {\bf 40} (1993), no. 2, 271--297.

\bibitem{Bon-22} J. Bonet, Weighted Banach spaces of analytic functions with sup-norms and operators between them: a survey, Rev. Real Acad. Cienc. Exactas Fis. Nat. Ser. A-Mat. (2022) 116:184.

\bibitem{BonDomLin-99} J. Bonet,  P. Dom\'anski and M. Lindstr\"om, Essential norm and weak compactness of composition operators on weighted spaces of analytic functions, Canad. Math. Bull. \textbf{42} (1999), no. 2, 139--148.
%
\bibitem{BonDomLinTas-98} J. Bonet, P. Dom\'anski, M. Lindstr\"om and J. Taskinen, Composition operators between weighted Banach spaces of analytic functions, J. Austral. Math. Soc. Ser. A \textbf{64} (1998), 101--118.

\bibitem{BonDomLin-01} J. Bonet, P. Domanski and M. Lindstr\"om, Weakly compact composition operators on weighted vector–valued Banach spaces of analytic mappings, Ann. Acad. Sci. Fenn. Ser. A I. Math. \textbf{26} (2001), 233--248. 


\bibitem{BotPelRue-07} G. Botelho, D. Pellegrino and P. Rueda, On composition ideals of multilinear mappings and homogeneous polynomials, Publ. Res. Inst. Math. Sci. \textbf{43} (2007), no. 4, 1139--1155.



\bibitem{CabJimRui-23} M. G. Cabrera-Padilla, A. Jim\'enez-Vargas and D. Ruiz-Casternado, On composition ideals and dual ideals of bounded holomorphic mappings, Results Math. (2023),  78:103.

\bibitem{DFJP-74} W. J. Davis, T. Figiel, W. B. Johnson and A. Pe\l czynski, Factoring weakly compact operators, J. Funct. Anal. \textbf{17} (1974), no. 3, 311--327.

\bibitem{DelPinSer-10} J. M. Delgado, C. Pi\~neiro and E. Serrano, Operators whose adjoints are quasi $p$-nuclear, Studia Math. \textbf{197} (2010), no. 3, 291--304. 

\bibitem{GalLasTur-12} D. Galicer, S. Lassalle and P. Turco, The ideal of $p$-compact operators: a tensor product approach, Studia Math. {\bf 211} (2012), no. 3, 269--286.

\bibitem{Gar-Mae-Rue-00} D. Garc\'ia, M. Maestre and P. Rueda, Weighted spaces of holomorphic functions on Banach spaces, Studia Math. {\bf 138} (2000), no. 1, 1--24.

\bibitem{g} A. Grothendieck, Produits tensoriels topologiques et espaces nucl\'{e}aires, Mem. Amer. Math. Soc. 16 (1955).

\bibitem{GupBaw-16} M. Gupta and D. Baweja, Weighted spaces of holomorphic functions on Banach spaces and the approximation property, Extracta Math. \textbf{31} (2016), no. 2, 123--144.

\bibitem{Jim-23} A. Jim\'enez-Vargas, On holomorphic mappings with relatively  $p$-compact range, Filomat \textbf{37} (2023), no. 24, 8067--8077.


\bibitem{JimRuiSep-23} A. Jim\'enez-Vargas, D. Ruiz-Casternado and J. M. Sepulcre, On holomorphic mappings with compact type range, Bull. Malays. Math. Sci. Soc. \textbf{46} (2023), no. 1, Paper No 20, 16 pp.

\bibitem{Jorda-13} E. Jorda, Weighted vector-valued holomorphic functions on Banach spaces, Abstr. Appl. Anal., Vol. 2013, Article ID 501592, Hindawi Publishing Corporation, 2013.

\bibitem{Ket-Inal} A. Keten \c{C}opur and R. \.{I}nal, Lipschitz operators associated with weakly $p$-compact and unconditionally $p$-compact sets, J. Adv. Res. Nat. Appl. Sci. (accepted).

\bibitem{Kim-14} J. M. Kim, Unconditionally $p$-null sequences and unconditionally $p$-compact operators, Studia Math \textbf{224} (2014), no. 2, 133--142, 

\bibitem{Kim-17} J. M. Kim, The ideal of unconditionally $p$-compact operators, Rocky Mountain J. Math. \textbf{47} (2017), no. 7, 2277--2293.

\bibitem{Kim-19} J. M. Kim, The ideal of weakly $p$-nuclear operators and its injective and surjective hulls, J. Korean Math. Soc. \textbf{56} (2019), no. 1, 225--237.

\bibitem{Kim-20} J. M. Kim, The ideal of weakly $p$-compact operators and its approximation property for Banach spaces, Ann. Acad. Sci. Fenn. Math \textbf{45} (2020), no. 2, 863--876.

\bibitem{LasTur-12} S. Lassalle and P. Turco, On $p$-compact mappings and the $p$-approximation properties, J. Math. Anal. Appl. \textbf{389} (2012), no. 2, 1204--1221.

\bibitem{Muj-91} J. Mujica, Linearization of bounded holomorphic mappings on Banach spaces, Trans. Amer. Math. Soc. \textbf{324} (1991), no. 2, 867--887.

\bibitem{Per-69} A. Persson, On some properties of $p$-nuclear and $p$-integral operators, Studia Math. {\bf 33} (1969), no. 2, 213--222.

\bibitem{PerPie-69} A. Persson and A. Pietsch, $p$-nuklear und $p$-integrale Abbildungen in Banachr\"aumen, Studia Math. {\bf 33} (1969), no. 1, 19--62.

\bibitem{Pie-80} A. Pietsch, Operator ideals, North-Holland Mathematical Library, vol. 20, North-Holland Publishing Co., Amsterdam-New York, 1980. Translated from German by the author.

\bibitem{Rub-Shie-70} L. A. Rubel and A. L. Shields, The second duals of certain spaces of analytic functions, J. Austral. Math. Soc. {\bf 11} (1970), 276--280.

\bibitem{SinKar-02} D. P. Sinha and A. K. Karn, Compact operators whose adjoints factor through subspaces of $\ell_p$, Studia Math. \textbf{150} (2002), no. 1, 17--33.

\bibitem{SinKar-08} D. P. Sinha and A. K. Karn, Compact operators which factor through subspaces of $\ell_p$, Math. Nachr. \textbf{3} (2008), no. 3, 412--423.

\end{thebibliography}
\end{document}